\documentclass{article}
\usepackage{amssymb,amsbsy}
\usepackage{amsmath}
\usepackage{amsthm}
\usepackage{mathrsfs}
\usepackage{color}
\usepackage[dvipsnames]{xcolor}
\usepackage{tikz}

\newtheorem{theorem}{Theorem}[section]
\newtheorem{lemma}[theorem]{Lemma}
\newtheorem{proposition}[theorem]{Proposition}
\newtheorem{corollary}[theorem]{Corollary}

\theoremstyle{definition}

\newtheorem{example}[theorem]{Example}

\theoremstyle{remark}
\newtheorem{remark}[theorem]{Remark}

\begin{document}
	
	\title{Numerical semigroups with monotone Ap\'ery set and fixed multiplicity and ratio}
	
	\author{Aureliano M. Robles-P\'erez\thanks{Departamento de Matem\'atica Aplicada \& Instituto de Matem\'aticas (IMAG), Universidad de Granada, 18071-Granada, Spain. \newline E-mail: \textbf{arobles@ugr.es} (\textit{corresponding author}); ORCID: \textbf{0000-0003-2596-1249}.}
		\mbox{ and} Jos\'e Carlos Rosales\thanks{Departamento de \'Algebra \& Instituto de Matem\'aticas (IMAG), Universidad de Granada, 18071-Granada, Spain. \newline E-mail: \textbf{jrosales@ugr.es}; ORCID: \textbf{0000-0003-3353-4335}.} }
	
	\date{\today}
	
	\maketitle
	
	\begin{abstract}
		We characterise the numerical semigroups with a monotone Ap\'ery set (MANS-semigroups for abbreviate). Moreover, we describe the families of MANS-semigroups when we set the multiplicity and the ratio.
	\end{abstract}
	\noindent {\bf Keywords:} Numerical semigroup, multiplicity, ratio, Frobenius number, monotone Ap\'ery set, suitably monotone element.
	
	\medskip
	
	\noindent{\it 2010 AMS Classification:} 20M14, 11D07.  	
	
	\section{Introduction}
	
	Let $ {Z}=\{0,\pm1,\pm2,\ldots\}$ the set of integer numbers and $\mathbb{N}=\{z\in\mathbb{Z} \mid z\geq0\}$. A \textit{numerical semigroup} is a subset $S$ of $\mathbb{N}$ such that it is closed under addition, $0\in S$, and $\mathbb{N}\setminus S = \{x\in\mathbb{N} \mid x\not\in S\}$ is finite.
	
	If $A$ is a non-empty subset of $\mathbb{N}$, then we denote by $\langle A \rangle$ the submonoid of $(\mathbb{N},+)$ generated by $A$, that is,
	\[ \langle A \rangle=\big\{\lambda_1a_1+\cdots+\lambda_na_n \mid n\in\mathbb{N}\setminus \{0\}, \ a_1,\ldots,a_n\in A, \ \lambda_1,\ldots,\lambda_n\in \mathbb{N}\big\}. \]
	From Lemma~2.1 of \cite{springer}, we have that $\langle A \rangle$ is a numerical semigroup if and only if $\gcd(A)=1$.
	
	If $S$ is a numerical semigroup and $S=\langle A \rangle$, then we say that $A$ is a \textit{system of generators of $S$}. Moreover, if $S\neq\langle B \rangle$ for any subset $B\subsetneq A$, then $A$ is a \textit{minimal system of generators for $S$}. From Theorem~2.7 of \cite{springer}, we have that each numerical semigroup admits a unique minimal system of generators and that such a system is finite. We denote by $\mathrm{msg}(S)$ the minimal system of generators of $S$. The cardinality of $\mathrm{msg}(S)$, denoted by $\mathrm{e}(S)$, is the \textit{embedding dimension of $S$}.
	
	If $S$ is a numerical semigroup, from the finiteness of $\mathbb{N}\setminus S$, then we can define two invariants of $S$. Namely, the \textit{Frobenius number of $S$} is the greatest integer that does not belong to $S$, denoted by $\mathrm{F}(S)$, and the \textit{genus of $S$} is the cardinality of $\mathbb{N}\setminus S$, denoted by $\mathrm{g}(S)$.
	
	The (extended) Frobenius problem (see \cite{alfonsin}) for a numerical semigroup $S$ consists of finding formulas to compute $\mathrm{F}(S)$ and $\mathrm{g}(S)$ in terms of $\mathrm{msg}(S)$. Its solution is well known for numerical semigroups with embedding dimension $\mathrm{e}(S)=2$ (see \cite{sylvester}). However, the problem is open for $e(S)\geq3$. In fact, in \cite{curtis}, it is proved that, in general, there is not possible to find polynomial formulas when $e(S)\geq3$.
	
	Let $S$ be a numerical semigroup and $n\in S\setminus\{0\}$. The \textit{Ap\'ery set of $n$ in $S$} (named so after \cite{apery}) is the set $\mathrm{Ap}(S,n)=\{s\in S \mid s-n\not\in S\}$. In Lemma~2.4 of \cite{springer}, it is shown that $\mathrm{Ap}(S,n)=\{w(0)=0, w(1), \ldots, w(n-1)\}$, where $w(i)$ is the least element of $S$ congruent with $i$ modulo $n$.
	
	Recall that if $S$ is a numerical semigroup, the least element of $S\setminus\{0\}$ (equivalently, the minimum of $\mathrm{msg}(S)$) is called the \textit{multiplicity of $S$}, denoted by $\mathrm{m}(S)$. Now, following the notation introduced in \cite{checa}, $S$ is a \textit{numerical semigroup with monotone Ap\'ery set} (\textit{MANS-semigroup} for abbreviate) if $S$ is a numerical semigroup fulfilling $w(1)<w(2)<\cdots<w(\mathrm{m}(S)-1)$, where $w(i)$ is the least element of $S$ congruent with $i$ modulo $\mathrm{m}(S)$.
	
	In \cite{checa}, the authors study some families of numerical semigroups with monotone Ap\'ery sets and fixed multiplicity. This work aims to characterise the family of MANS-semigroups.
	
	Firstly, in Section~\ref{sect-two-dim}, we see a necessary condition that, in particular, it is sufficient for the two-embedding dimensional case.
	
	Then, in Section~\ref{sect-three-dim}, we analyse the MANS-semigroups with embedding dimension equal to three in detail. Thus, in Subsection~\ref{generators}, we characterise the tuples $(n_1, n_2, n_3)$ such that $\langle\{n_1,n_2,n_3\}\rangle$ is a MANS-semigroup. Furthermore, we solve the (extended) Frobenius problem for those semigroups in Subsection~\ref{frobenius} and, in Subsections~\ref{pseudofrobenius} and \ref{irreducible}, we study pseudo-Frobenius numbers and MANS-semigroups, with embedding dimension equal to three, that are irreducible (recall that a numerical semigroup is irreducible if it cannot be expressed as the intersection of two numerical semigroups containing it properly).
	
	Finally, in Section~\ref{sect-general}, we study the MANS-semigroups with arbitrary embedding dimension, giving a characterisation of them and, in Subsection~\ref{tree}, describing the tree associated with the family of numerical semigroups with fixed multiplicity and ratio (that is, when we set the two least minimal generators).

	\section{Two-embedding dimensional case}\label{sect-two-dim}
	
	If $S$ is a numerical semigroup and $\mathrm{msg}(S)=\{n_1<n_2<\cdots<n_e\}$, then $\mathrm{m}(S)=n_1$, $\mathrm{r}(S)=n_2$ and $\mathrm{M}(S)=n_e$ are the multiplicity, the \textit{ratio} and the \textit{maximum minimal generator of $S$}, respectively. In particular, if $S$ is a two-dimensional numerical semigroup, then $\mathrm{msg}(S)=\{\mathrm{m}(S) < \mathrm{r}(S)\}$ and $ \mathrm{r}(S) = \mathrm{M}(S)$.
	
	Let us see two necessary conditions for MANS-semigroups. The first is a direct consequence of a well-known fact: $\{n_2,\ldots,n_e\} \subseteq\mathrm{Ap}(S,n_1)$.
	
	\begin{lemma}\label{lem-000}
		Let $S=\langle n_1<n_2<\cdots<n_e \rangle $ be a numerical semigroup with $n_1\geq 2$. If $S$ is a MANS-semigroup, then $n_2 \bmod n_1 < n_3 \bmod n_1 < \cdots < n_e \bmod n_1$.
	\end{lemma}
	
	\begin{lemma}\label{lem-001}
		Let $S$ be a numerical semigroup with $\mathrm{m}(S)\geq 2$. If $S$ is a MANS-semigroup, then there exists $a\in\mathbb{N}\setminus\{0\}$ such that $\mathrm{r}(S)=a\mathrm{m}(S)+1$.
	\end{lemma}
	
	\begin{proof}
		From the definitions, we have that $\mathrm{r}(S)=\min(\mathrm{Ap}(S,\mathrm{m}(S)) \setminus \{0\})$, and since $S$ is a MANS-semigroup, we deduce there exists $a\in\mathbb{N}\setminus\{0\}$ such that $\mathrm{r}(S)=a\mathrm{m}(S)+1$.
	\end{proof}

	In the two-embedding dimensional case, the necessary condition of the above lemma is also sufficient.
	
	\begin{proposition}\label{prop-002}
		Let $S=\langle n_1 , n_2 \rangle$ be a numerical semigroup with $2\leq n_1<n_2$. Then $S$ is a MANS-semigroup if and only if there exists $a\in\mathbb{N}\setminus\{0\}$ such that $n_2=an_1+1$.
	\end{proposition}
	
	\begin{proof}
		Necessity is just Lemma~\ref{lem-001}. For sufficiency, it is easy to check that $\mathrm{Ap}(S,n_1)=\{0 < n_2 < 2n_2 < \cdots < (n_1-1)n_2\}$ and $in_2\equiv i \pmod{n_1}$ for all $i\in\{1,\ldots,n_1-1\}$.
	\end{proof}

	\section{Three-embedding dimensional case}\label{sect-three-dim}
	
	Observe that if $S$ is a numerical semigroup with embedding dimension equal to three, then $\mathrm{msg}(S)=\{\mathrm{m}(S) < \mathrm{r}(S) < \mathrm{M}(S)\}$.
	
	\subsection{Minimal generators}\label{generators}
	
	In the following result, we show a necessary condition for MANS-semigroups with embedding dimension equal to three.
	
	\begin{proposition}\label{prop-01}
		If $S$ is a MANS-semigroup and $\mathrm{e}(S)=3$, then there exists $\{m,a,b,t\}\subseteq\mathbb{N}$ such that $m\geq 3$, $a\geq 1$, $t\in\{2,\ldots,m-1\}$, $(t-1)(am+1)<bm+t<t(am+1)$, and $(\mathrm{m}(S), \mathrm{r}(S), \mathrm{M}(S)) = (m, am+1, bm+t)$.
	\end{proposition}
	
	\begin{proof}
		From Proposition~2.10 of \cite{springer}, we know that $\mathrm{e}(S)\leq \mathrm{m}(S)$ and, therefore, $\mathrm{m}(S)\geq 3$.
		
		By applying Lemma~\ref{lem-001}, there exists $a\in\mathbb{N}\setminus\{0\}$ such that $\mathrm{r}(S)=a\mathrm{m}(S)+1$. Moreover, we have that $\mathrm{M}(S)=b\mathrm{m}(S)+t$ with $b\in\mathbb{N}\setminus\{0\}$ and $t\in\{2,\ldots,m-1\}$.
		
		If $\mathrm{Ap}(S,\mathrm{m}(S)) = \{w(0)=0, w(1), \ldots, w(\mathrm{m}(S)-1)\}$, then clearly $w(t)=\mathrm{M}(S)$ and, since $S$ is a MANS-semigroup, we have that $w(0)=0 < w(1) = a\mathrm{m}(S)+1 < w(2) = 2(a\mathrm{m}(S)+1) < \cdots < w(t-1) = (t-1)(a\mathrm{m}(S)+1) < w(t) = \mathrm{M}(S) = b\mathrm{m}(S)+t$. Finally, since $b\mathrm{m}(S)+t \not\in \langle \mathrm{m}(S), a\mathrm{m}(S)+1 \rangle$, we deduce that $b\mathrm{m}(S)+t < t(a\mathrm{m}(S)+1)$.
	\end{proof}
	
	Our next aim is to show that the condition given in Proposition~\ref{prop-01} is also sufficient. From here, $m,a,b,t$ are positive integers such that $m\geq 3$, $a\geq 1$, $t\in\{2,\ldots,m-1\}$, $(t-1)(am+1)<bm+t<t(am+1)$; moreover, we consider $S = \langle m, am+1, bm+t \rangle$. Let us see that $S$ is a three-embedding dimension numerical semigroup, and then let us describe the Ap\'ery set $\mathrm{Ap}(S,m)$. 
	
	\begin{lemma}\label{lem-02}
		$S$ is a numerical semigroup with $\mathrm{e}(S)=3$.
	\end{lemma}

	\begin{proof}
		Since $\gcd\{m,am+1\}=1$, we have that $S$ is a numerical semigroup. Moreover, since $(t-1)(am+1)<bm+t$, then $am+1<bm+t$. Finally, to prove that $\mathrm{e}(S)=3$, it suffices to see that $mb+t\not\in\langle a, am+1 \rangle$, which is true because $bm+t<t(am+1)$.
	\end{proof}
	
	Three previous results, in which we delimit the possible elements, are necessary to explicitly give the Ap\'ery set $\mathrm{Ap}(S,m)$.
	
	\begin{lemma}\label{lem-03}
		Let $\lambda \in\mathbb{N}\setminus\{0\}$ such that $(\lambda-1)t<m\leq \lambda t$. Then $\lambda (bm+t) \not\in \mathrm{Ap}(S,m)$.
	\end{lemma}
	
	\begin{proof}
		It follows directly from the hypothesis that $\lambda t-m\in\{0,\ldots,t-1\}$. Moreover, we have that $\lambda(bm+t)\equiv(\lambda t-m)(am+1)\pmod{m}$. Then, since $(\lambda t-m)(am+1) \leq (t-1)(am+1)<bm+t\leq \lambda(bm-t)$, we deduce that $\lambda(bm+t) = (\lambda t-m)(am+1) +\mu m$ for some $\mu\in\mathbb{N}\setminus\{0\}$. Therefore, $\lambda (bm+t) \not\in \mathrm{Ap}(S,m)$.
	\end{proof}
	
	An immediate consequence of the above lemma is the following one.
	
	\begin{lemma}\label{lem-04}
		If $\lambda \in\mathbb{N}$ and $\lambda t\geq m$, then $\lambda (bm+t) \not\in \mathrm{Ap}(S,m)$.
	\end{lemma}
	
	\begin{lemma}\label{lem-05}
		If $\mu \in\mathbb{N}$ and $\mu\geq t$, then $\mu (am+1) \not\in \mathrm{Ap}(S,m)$.
	\end{lemma}
	
	\begin{proof}
		It is clear that $\mu (am+1)\equiv (bm+t)+(\mu-t)(am+1) \pmod{m}$. Moreover, $(bm+t)+(\mu-t)(am+1)\in S$ and $(bm+t)+(\mu-t)(am+1) < \mu(am+1)$ (because $bm+t < t(am+1)$). Therefore, there exists $\lambda \in\mathbb{N}\setminus\{0\}$ such that $\mu(am+1) = (bm+t)+(\mu-t)(am+1) +\lambda m$. In consequence, $\mu (am+1) \not\in \mathrm{Ap}(S,m)$.
	\end{proof}
	
	We are ready to show $\mathrm{Ap}(S,m)$. As usual, for $x\in\mathbb{R}$ and $a,b\in\mathbb{N}$, we denote by $\lfloor x \rfloor = \max\{n\in\mathbb{N} \mid n\leq x\}$ and by $a \bmod b$ the remainder of the division of $a$ by $b$.
	
	\begin{proposition}\label{prop-06}
		If $\mathrm{Ap}(S,m) = \{ w(0), w(1), \ldots, w(m-1)\}$, then
		\[ w(i) = \bigg\lfloor \frac{i}{t} \bigg\rfloor (bm+t) + (i \bmod t)(am+1), \; i\in\{0,1,\ldots,m-1\}. \]
	\end{proposition}
	
	\begin{proof}
		By Lemmas~\ref{lem-04} and \ref{lem-05}, we deduce that $w(i)=\lambda(bm+t)+\mu(am+1)$ for some $\{\lambda,\mu\} \subseteq \mathbb{N}$ such that $\lambda t<m$ and $\mu<t$. Then, since $\lambda t+\mu \equiv i \pmod{m}$ and $i=\lfloor \frac{i}{t} \rfloor t + (i\bmod t)$, we can conclude that $\lambda=\lfloor \frac{i}{t} \rfloor$ and $\mu=i \bmod t$.
	\end{proof}
	
	We end this section with the characterisation of the three-embedding dimension MANS-semigroups.
	
	\begin{theorem}\label{thm-07}
		The following conditions are equivalent.
		\begin{enumerate}
			\item $S$ is a MANS-semigroup with $\mathrm{e}(S)=3$.
			\item $S = \langle m, am+1, bm+t \rangle$, where $\{m,a,b,t\}\subseteq\mathbb{N}$, $m\geq 3$, $a\geq 1$, $t\in\{2,\ldots,m-1\}$, and $(t-1)(am+1)<bm+t<t(am+1).$
		\end{enumerate}
	\end{theorem}

	\begin{proof}
		($1. \Rightarrow 2.$) This is Proposition~\ref{prop-01}.
		
		($2. \Rightarrow 1.$) Form Lemma~\ref{lem-02}, we know that $S$ is a numerical semigroup with $\mathrm{e}(S)=3$. To finish the proof, it will be enough to see that if $\mathrm{Ap}(S,n)=\{w(0), w(1), \ldots, w(m-1)\}$, then $w(i)<w(i+1)$ for all $i\in\{0,\ldots,m-2\}$. On the one side, if $(i+1)\bmod{t} > i\bmod{t}$, then we deduce that $w(i)<w(i+1)$ by applying Proposition~\ref{prop-06}. On the other side, if $(i+1)\bmod{t} \leq i\bmod{t}$, then $(i+1)\bmod{t}=0$ and, thereby, $\lfloor \frac{i+1}{t} \rfloor = \lfloor \frac{i}{t} \rfloor-1$. By applying again Proposition~\ref{prop-06}, we have that $w(i)<w(i+1)$.
	\end{proof}

	\subsection{Frobenius problem}\label{frobenius}
	
	The following result is the first part of Proposition~2.12 in \cite{springer}.
	
	\begin{proposition}\label{prop-08}
		If $S$ is a numerical semigroup and $n\in S\setminus\{0\}$, then $\mathrm{F}(S) = \max(\mathrm{Ap}(S,n))-n$.
	\end{proposition}
	
	An immediate consequence of Propositions~\ref{prop-06} and \ref{prop-08} is the next.
	
	\begin{proposition}\label{prop-09}
		If $S= \langle m,am+1,bm+t\rangle$ is a three-embedding dimension MANS-semigroup, then $\mathrm{F}(S) =  r(am+1)+q(bm+t)-m$, where $q=\left\lfloor \frac{m-1}{t} \right\rfloor$ and $r=(m-1)\bmod{t}$.
	\end{proposition}
	
	\begin{remark}\label{rem-09}
		By applying the second case of the first theorem in \cite{byrnes}, we recover Proposition~\ref{prop-09}. 
	\end{remark}
	
	Let us see an example of the above proposition.
	
	\begin{example}\label{exmp-10}
		Let $S=\langle 5,6,13 \rangle = \{0,5,6,10,11,12,13,15,\to\}$ (where the symbol $\to$ means that all integers greater than 15 belong to $S$). Then $\mathrm{Ap}(S,5)=\{w(0)=0, w(1)=6, w(2)=12, w(3)=13, w(4)=19\}$ and, therefore, $S$ is a MANS-semigroup with $\mathrm{e}(S)=3$. Moreover, since $m=5$, $am+1=6$, $bm+t=13$, and $t=3$, then $q=1$, $r=1$, and consequently $\mathrm{F}(S) = 6+13-5 = 14$.
	\end{example}
	
	The following result is the second statement of Proposition~2.12 in \cite{springer}.

	\begin{proposition}\label{prop-11}
		Let $S$ be a numerical semigroup, $n\in S\setminus\{0\}$, and $\mathrm{Ap}(S,n)=\{w(0),w(1),\ldots,w(n-1)\}$. Then $\mathrm{g}(S)=\frac{w(0)+w(1)+\cdots+w(n-1)}{n}-\frac{n-1}{2}$.
	\end{proposition}
	
	Now we can show a formula for the genus of a three-embedding dimension MANS-semigroup.

	\begin{proposition}\label{prop-12}
		If $S= \langle m,am+1,bm+t\rangle$ is a three-embedding dimension MANS-semigroup, then
		\[ \mathrm{g}(S)=\frac{qt(t-1)+r(r+1)}{2m}(am+1) + \frac{qt(q-1)+2q(r+1)}{2m}(bm+t) - \frac{m-1}{2}, \] where $q=\left\lfloor \frac{m-1}{t} \right\rfloor$ and $r=(m-1)\bmod{t}$. 
	\end{proposition}
	
	\begin{proof}
		As a consequence of Proposition~\ref{prop-06}, we have that
		\[ \begin{split} 
			\mathrm{Ap}(S,m) & = \{ 0,(am+1),\ldots, (t-1)(am+1), \\
			 & (bm+t),(am+1)+(bm+t),\ldots, (t-1)(am+1)+(bm+t),\ldots, \\
			 & (q-1)(bm+t),(am+1)+(q-1)(bm+t), \ldots, \\
			 & (t-1)(am+1)+(q-1)(bm+t), \\
			 & r(bm+t),(am+1)+q(bm+t),\ldots, r(am+1)+q(bm+t) \}
			\end{split} \]
		Then, by applying Proposition~\ref{prop-11}, we get the result.
	\end{proof}
	
	\begin{remark}
		Since $qt(t-1)+r(r+1) + (qt(q-1)+2q(r+1))t = m(m-1)$, we can rewrite the formula of the previous proposition as
		\[ \mathrm{g}(S)=\frac{qt(t-1)+r(r+1)}{2}a + \frac{qt(q-1)+2q(r+1)}{2}b. \]
	\end{remark}
	
	Let us see two examples of the content of the above proposition.
	
	\begin{example}\label{exmp-13}
		Let $S=\langle 5,6,13 \rangle$ the numerical semigroup of Example~\ref{exmp-10}. Then $a=1$, $b=2$, $m=5$, $t=3$, $q=1$, and $r=1$. By applying Proposition~\ref{prop-12}, we have that $\mathrm{g}(S) = \frac{(3\times2+1\times2)\times6+(3\times0+2\times1\times2)\times13}{2\times5}-\frac{4}{2}=8 $.
	\end{example}
	
	\begin{example}\label{exmp-14}
		If $S=\langle 10,11,23 \rangle$, then $\mathrm{Ap}(S,10)=\{w(0)=0, w(1)= 11, w(2)=22, w(3)=23, w(4)=34, w(5)=45, w(6)=46, w(7)=57, w(8)=68, w(9)=69\}$ and, therefore, $S$ is a MANS-semigroup. Moreover, since $a=1$, $b=2$, $m=10$, $t=3$, $q=3$, and $r=0$, from Proposition~\ref{prop-12}, it follows that $\mathrm{g}(S) = \frac{(9\times2+0\times1)\times11+(9\times2+2\times3\times1)\times23}{2\times10}-\frac{9}{2}=33 $.
	\end{example}

	\subsection{Pseudo-Frobenius numbers}\label{pseudofrobenius}
	
	Let $S$ be a numerical semigroup. Following the terminology in \cite{JPAA}, a \textit{pseudo-Frobenius number of $S$} is an element $x\in\mathbb{Z}\setminus S$ such that $x+s\in S$ for all $s\in S\setminus\{0\}$. We denote by $\mathrm{PF}(S) = \{ x \mid x \mbox{ is a pseudo-Frobenius number of $S$} \}$. The cardinality of $\mathrm{PF}(S)$ is called the \textit{type of $S$}, denoted by $\mathrm{t}(S)$. From \cite{froberg}, we have that if $S$ is a numerical semigroup with $\mathrm{e}(S)=3$, then $\mathrm{t}(S)\in\{1,2\}$.
	
	Let $S$ be a numerical semigroup. We define over $\mathbb{Z}$ the following binary relation: $a\leq_Sb$ if $b-a\in S$. It is clear that $\leq_S$ is a non-strict partial order relation (that is, it is reflexive, transitive, and anti-symmetric).
	
	The following result is Proposition~2.20 of \cite{springer} (see also Proposition~7 of \cite{froberg}) and characterises the pseudo-Frobenius numbers in terms of the maximal elements of $\mathrm{Ap}(S,n)$ with respect to the relation $\leq_S$.
	
	\begin{proposition}\label{prop-15}
		Let $S$ be a numerical semigroup and $n\in S\setminus\{0\}$. Then 
		\[\mathrm{PF}(S)=\{w-n \mid w\in \mathrm{Maximals}_{\leq_S} (\mathrm{Ap}(S,n))  \}.\]
	\end{proposition}

	Before continuing, let us see two examples.
	
	\begin{example}\label{exmp-16a}
		Let $S=\langle 5,6,13 \rangle$ as in Example~\ref{exmp-10}. Then $\mathrm{Ap}(S,5)=\{0,6,12,13,19\}$ and, thereby, $\mathrm{Maximals}_{\leq_S} (\mathrm{Ap}(S,5)) = \{12,19\}$. By applying Proposition~\ref{prop-15}, we have that $\mathrm{PF}(S) = \{7,14\}$.
	\end{example}
	
	\begin{example}\label{exmp-16b}
		Let $S=\langle 10,11,23 \rangle$ as in Example~\ref{exmp-14}. Then $\mathrm{Ap}(S,10)=\{0,11,22,23,34,45,46,57,68,69\}$ and $\mathrm{Maximals}_{\leq_S} (\mathrm{Ap}(S,10)) = \{68,69\}$. By Proposition~\ref{prop-15}, we have that $\mathrm{PF}(S) = \{58,59\}$.
	\end{example}
	
	The following result follows from the proof of the Proposition~\ref{prop-12}.
	
	\begin{lemma}\label{lem-17}
		Let $S= \langle m,am+1,bm+t\rangle$ be a three-embedding dimension MANS-semigroup, $q=\left\lfloor \frac{m-1}{t} \right\rfloor$, and $r=(m-1)\bmod{t}$. Then $ \{q(bm+t)+r(am+1)\} \subseteq \mathrm{Maximals}_{\leq_S} (\mathrm{Ap}(S,m)) \subseteq \{(q-1)(bm+t)+(t-1)(am+1), q(bm+t)+r(am+1)\} $
		and, consequently, $ \{q(bm+t)+r(am+1)-m\} \subseteq \mathrm{PF}(S) \subseteq \{(q-1)(bm+t)+(t-1)(am+1)-m, q(bm+t)+r(am+1)-m\}$.		
	\end{lemma}
	
	Let us characterise when a three-embedding dimension MANS-semigroup has a type equal to one or when it equal to two.
	
	\begin{proposition}\label{prop-18}
		Let $S= \langle m,am+1,bm+t\rangle$ be a three-embedding dimension MANS-semigroup, $q=\left\lfloor \frac{m-1}{t} \right\rfloor$, and $r=(m-1)\bmod{t}$. Then $\mathrm{t}(S)=1$ if and only if $t \,\vert\, m$ (that is, $t$ divides $m$).
	\end{proposition}
	
	\begin{proof}
		\textit{(Necessity.)} From Proposition~\ref{prop-15} and Lemma~\ref{lem-17}, we deduce that if $\mathrm{t}(S)=1$, then $q(bm+t)+r(am+1)-((q-1)(bm+t)+(t-1)(am+1))\in S$ and, therefore, $bm+t+(r-t+1)(am+1)\in S$. By applying that $bm+t\in\mathrm{msg}(S)$, we have that $r-t+1\geq 0$ and, consequently, $r=t-1$. Thus, $m-1=qt+t-1$ and, thereby, $t \,\vert\, m$.
		
		\textit{(Sufficiency.)} If $t \,\vert\, m$, then there exists $k\in\mathbb{N}$ such that $m=kt$ and, therefore, $m-1=(k-1)t+t-1$. Consequently, $r=t-1$ and  $q(bm+t)+r(am+1)-((q-1)(bm+t)+(t-1)(am+1))=bm+t\in S$. Now, by applying Proposition~\ref{prop-15} and Lemma~\ref{lem-17}, we conclude that $\mathrm{t}(S)=1$.
	\end{proof}
	
	We deduce the following result from Propositions~\ref{prop-15} and \ref{prop-18} and Lemma~\ref{lem-17}. We denote by $t\nmid m$ that $t$ does not divide $m$.
		
	\begin{proposition}\label{prop-21}
		Let $S= \langle m,am+1,bm+t\rangle$ be a three-embedding dimension MANS-semigroup, $q=\left\lfloor \frac{m-1}{t} \right\rfloor$, and $r=(m-1)\bmod{t}$.
		\begin{enumerate}
			\item If $t \,\vert\, m$, then $\mathrm{PF}(S) = \{q(bm+t)+r(am+1)-m\}$.
			\item If $t\nmid m$, then $\mathrm{PF}(S) = \{(q-1)(bm+t)+(t-1)(am+1)-m, q(bm+t)+r(am+1)-m\}$.
		\end{enumerate}
	\end{proposition}
	
	\begin{example}\label{exmp-19}
		Let $S=\langle 6,7,15 \rangle$. Then $\mathrm{Ap}(S,6)=\{ w(0)=0, w(1)=7, w(2)=14, w(3)=15, w(4)= 22, w(5)=29\}$. Therefore, $S$ is a MANS-semigroup with $\mathrm{e}(S)=3$. Since $m=6$ and $r=3$, from Proposition~\ref{prop-18} we can assert that $\mathrm{t}(S)=1$; indeed, $\mathrm{PF}(S)=\{23\}$.
	\end{example}

	\subsection{Irreducibility}\label{irreducible}
	
	Recall that a numerical semigroup $S$ is \textit{irreducible} if it is not expressible as the intersection of two numerical semigroups properly containing $S$. This concept was introduced in \cite{pacific}, where it is shown that a numerical semigroup $S$ is irreducible if and only if it is maximal (with respect to the inclusion order) in the set formed by all numerical semigroups with Frobenius number equal to $\mathrm{F}(S)$. From \cite{barucci} and \cite{froberg}, it follows that the family of irreducible numerical semigroups is the union of two well-known families, the symmetric numerical semigroups and the pseudo-symmetric numerical semigroups (see \cite{pacific}). Furthermore, a numerical semigroup is symmetric (pseudo-symmetric, respectively) if it is irreducible and has an odd Frobenius number (even Frobenius number, respectively).
	
	The following result is consequence of Corollaries~4.5, 4.11 and 4.16 in \cite{springer}.
	
	\begin{proposition}\label{prop-20}
		Let $S$ be a numerical semigroup.
		\begin{enumerate}
			\item $S$ is symmetric if and only if $\mathrm{t}(S)=1$ (equivalently, $\mathrm{PF}(S) = \{\mathrm{F}(S)\}$).
			\item $S$ is symmetric if and only if $\mathrm{F}(S)=2\mathrm{g}(S)-1$.
			\item $S$ is pseudo-symmetric if and only if $\mathrm{PF}(S)=\left\{\frac{\mathrm{F}(S)}{2},\mathrm{F}(S)\right\}$.
			\item $S$ is pseudo-symmetric if and only if $\mathrm{F}(S)=2\mathrm{g}(S)-2$.
		\end{enumerate}
	\end{proposition}
	
	Let us observe that Proposition~\ref{prop-18} characterises the MANS-semigroups with embedding dimension equal to three that are symmetric. Note also that, from Example~\ref{exmp-16a} and Proposition~\ref{prop-20}, we know that $S=\langle 5,6,13 \rangle$ is a MANS-semigroup with embedding dimension equal to three that is pseudo-symmetric. We now propose to characterise this class of semigroups.
	
	\begin{proposition}\label{prop-22}
		Let $S= \langle m,am+1,bm+t\rangle$ be a three-embedding dimension MANS-semigroup. Then $S$ is pseudo-symmetric if and only if $t=\frac{m+1}{2}$ and $t=\frac{b+1}{a}$.
	\end{proposition}
	
	\begin{proof}
		From Propositions~\ref{prop-09}, \ref{prop-12}, and \ref{prop-20}, $S$ is pseudo-symmetric if and only if $r(am+1)+q(bm+t)$ is equal to
		\[ \frac{qt(t-1)+r(r+1)}{m}(am+1) + \frac{qt(q-1)+2q(r+1)}{m}(bm+t) -1, \]
		where $q=\left\lfloor \frac{m-1}{t} \right\rfloor$ and $r=(m-1)\bmod{t}$.
		Since
		\[ (mr-qt(t-1)-r(r+1))(am+1) + (qm-qt(q-1)-2q(r+1))(bm+t) = \]
		\[ ((m-r-1)r-qt(t-1))(am+1) + q(m-t(q-1)-2(r+1))(bm+t) = \]
		\[ qt(r-t+1)(am+1) + q(t-1-r)(bm+t) = q(t-1-r)(ta-b)(-m), \]
		we deduce that $S$ is pseudo-symmetric if and only if $q(t-1-r)(ta-b)=1$. Finally, since $q, t-1-r, ta-b\in\mathbb{N}$ (recall that $bm+t<t(am+1)$), we deduce that $q=t-1-r=ta-b=1$ and we get the result by observing that
		\begin{itemize}
			\item $t=\frac{m+1}{2}$ if and only if $q=1$ and $r=t-2$;
			\item $t=\frac{b+1}{a}$ if and only if $ta-b=1$. \qedhere
		\end{itemize} 
	\end{proof}
	
	\begin{remark}
		With similar reasoning as in the above proof, we can recover Proposition~\ref{prop-18} using condition 2 of Proposition~\ref{prop-20}.
	\end{remark}
	
	Let us see some examples relate to the above proposition. In particular, from the first three, we conclude that conditions $t=\frac{b+1}{a}$ and $t=\frac{m+1}{2}$ are independent.
	
	\begin{example}\label{exmp-23a}
		Let $S=\langle 5,6,19 \rangle$. Then $m=5$, $a=1$, $b=3$, and $t=4$. Thus, $t=\frac{b+1}{a}$ and $t\neq\frac{m+1}{2}$ ($q=1$, but $r=0\neq t-2$). Note that $S$ is a MANS-semigroup and, since $\mathrm{PF}(S)=\{13,14\}$, it is not pseudo-symmetric.
	\end{example}
	
	\begin{example}\label{exmp-23b}
		Let $S=\langle 5,11,17 \rangle$. Then $m=5$, $a=2$, $b=3$, and $t=2$. Thus, $t=\frac{b+1}{a}$ and $t\neq\frac{m+1}{2}$ ($r=0=t-2$, but $q=2\neq1$). We can easily check that $S$ is a MANS-semigroup and, since $\mathrm{PF}(S)=\{23,29\}$, it is not pseudo-symmetric.
	\end{example}

	\begin{example}\label{exmp-23c}
		Let $S=\langle 5,11,23 \rangle$. Then $m=5$, $a=2$, $b=4$, and $t=3$. Therefore, $t\neq\frac{b+1}{a}$ and $t=\frac{m+1}{2}$. We have that $S$ is a MANS-semigroup and, since $\mathrm{PF}(S)=\{17,29\}$, it is not pseudo-symmetric.
	\end{example}

	\begin{example}\label{exmp-23d}
		Let $S=\langle 5,11,28 \rangle$. Then $m=5$, $a=2$, $b=5$, and $t=3$. Since $t=\frac{b+1}{a}=\frac{m+1}{2}$, $S$ is a pseudo-symmetric MANS-semigroup (note that  $\mathrm{PF}(S)=\{17,34\}$).
	\end{example}

	\begin{example}\label{exmp-23e}
		If $m=3$, then the symmetric MANS-semigroups are of the form $S=\langle 3,3a+1 \rangle$ and the pseudo-symmetric MANS-semigroups of the form $S=\langle 3,3a+1,6a-1 \rangle$, with $a\in\mathbb{N}\setminus\{0\}$ in both cases. Note that, for each $a\in\mathbb{N}\setminus\{0\}$, $6a-1$ is the Frobenius number of $\langle 3,3a+1 \rangle$. 
	\end{example}

	\section{General case}\label{sect-general}
	
	In this section, we analyse the general case of MANS-semigroups, that is, we consider numerical semigroups of arbitrary embedding dimension. 
	
	As stated in Section~\ref{sect-two-dim}, if $S$ is a numerical semigroup with $\mathrm{msg}(S)=\{n_1<n_2<\cdots<n_e\}$, then $\mathrm{m}(S)=n_1$, $\mathrm{r}(S)=n_2$, and $\mathrm{M}(S)=n_e$ are the multiplicity, the ratio, and the greatest minimal generator of $S$, respectively. Moreover, if $S$ is a MANS-semigroup, then $r(S)=am(S)+1$ for some $a\in\mathbb{N}\setminus\{0\}$.
	
	Let us first give a characterisation of MANS-semigroups. We start by seeing how we can add to a MANS-semigroup $S$ a new minimal generator (greater than $\mathrm{M}(S)$) so that we obtain a new MANS-semigroup (with a higher embedding dimension).
	
	\begin{lemma}\label{lem-24}
		Let $S$ be a MANS-semigroup with $\mathrm{msg}(S)=\{n_1<n_2<\cdots<n_e\}$ ($2\leq e\leq n_1-1$) and $\mathrm{Ap}(S,n_1)=\{w(0), w(1), \ldots, w(n_1-1)\}$. If $n_{e+1}\in\mathbb{N}$ fulfil that $n_e<n_{e+1}$, $n_e \bmod n_1 < n_{e+1} \bmod n_1$, and $w(n_{e+1} \bmod n_1-1)<n_{e+1}<w(n_{e+1}\bmod n_1)$, then $S'=\langle n_1,\ldots,n_e,n_{e+1} \rangle$ is a MANS-semigroup with $\mathrm{e}(S')=\mathrm{e}(S)+1$.
	\end{lemma}
	
	\begin{proof}
		From condition $n_{e+1}<w(n_{e+1}\bmod n_1)$, we deduce that  $n_{e+1}\not\in S$ and, since $n_e<n_{e+1}$, then $\mathrm{msg}(S')=\{n_1<n_2<\cdots<n_e<n_{e+1}\}$. Therefore, $\mathrm{e}(S')=\mathrm{e}(S)+1$.
		
		Let $\mathrm{Ap}(S',n_1)=\{w'(0), w'(1), \ldots, w'(n_1-1)\}$. By the construction of $S'$, it is clear that $w'(i)\leq w(i)$ for all $i\in\{0,1,\ldots,n_1-1\}$. To prove that $S'$ is a MANS-semigroup, we analyse what happens between two consecutive elements of $\mathrm{Ap}(S',n_1)$. We will consider two cases, taking $i\geq 1$.
		\begin{enumerate}
			\item If $w'(i) \in S$, then $w'(i)=w(i)$ and, therefore, $w'(i-1)\leq w(i-1)< w(i) = w'(i)$.
			\item If $w'(i) \not\in S$, then $w'(i)=kn_{e+1}+w(j)$ for some $k\in\mathbb{N}\setminus\{0\}$ and some $j\in\{0,\ldots,n_1-1\}$. Once again, we distinguish two cases.
			\begin{enumerate}
				\item If $j\neq 0$, then $w'(i-1) \leq kn_{e+1}+w(j-1) < kn_{e+1}+w(j) = w'(i)$.
				\item If $j=0$, the $w'(i)=kn_{e+1}$ and, consequently, $w'(i-1) \leq (k-1)n_{e+1}+w(n_{e+1}\bmod n_1-1) < kn_{e+1} = w'(i)$. \qedhere
			\end{enumerate}
		\end{enumerate}
	\end{proof}
	
	Let us now see that if we remove the greatest minimal generator of a MANS-semigroup $S$, we get a new MANS-semigroup (with a less embedding dimension).
	
	\begin{lemma}\label{lem-25}
		Let $S$ be a MANS-semigroup with $\mathrm{msg}(S)=\{n_1<n_2<\cdots<n_e<n_{e+1}\}$ ($e\geq 2$). Then $S'= \langle n_1,n_2,\ldots,n_e \rangle$ is a MANS-semigroup with $\mathrm{e}(S')=\mathrm{e}(S)-1$.
	\end{lemma}
	
	\begin{proof}
		The equality $\mathrm{e}(S')=\mathrm{e}(S)-1$ is trivial by the construction of $S'$.
		
		From Lemma~\ref{lem-001}, we know that if $e=2$, then $S' = \langle n_1,n_2 \rangle = \langle n_1,an_1+1 \rangle$ for some $a\in\mathbb{N}\setminus\{0\}$. By Proposition~\ref{prop-002}, $S'$ is a MANS-semigroup.
		
		Now, let $e\geq 3$, $\mathrm{Ap}(S,n_1)=\{w(0), w(1), \ldots, w(n_1-1)\}$ and $\mathrm{Ap}(S',n_1)=\{w'(0), w'(1), \ldots, w'(n_1-1)\}$. To prove that $S'$ is a MANS-semigroup, we analyse what happens between two consecutive elements of $\mathrm{Ap}(S',n_1)$. 
		
		Let $i\geq 2$. It is clear that there exists $k\in\{2,3,\ldots,e\}$ such that $w'(i)=n_k+w'(j)$, where $j=(n_k-i)\bmod n_1$. Thus, $n_k \bmod n_1 -1 + j \equiv (i-1)\pmod{n_1}$ and, therefore, $w'(i-1) \leq w'(n_k \bmod n_1 -1) + w'(j)$. Now, since $S$ is MANS-semigroup, then $w'(n_k \bmod n_1 -1) < w'(n_k \bmod n_1)=n_k$. In conclusion, $w'(i-1) \leq w'(j)+w'(n_k \bmod n_1 -1) < w'(j)+n_k = w'(i)$.
	\end{proof}
	
	The following result follows immediately from the above lemma.
	
	\begin{corollary}\label{cor-26}
		Let $S$ be a MANS-semigroup with $\mathrm{msg}(S)=\{n_1<n_2<\cdots<n_e<n_{e+1}\}$ ($e\geq 2$). Then $S'= \langle n_1,n_2,\ldots,n_i \rangle$ is a MANS-semigroup with $\mathrm{e}(S')=i$ for all $i\in\{2,\ldots,e\}$.
	\end{corollary}
	
	We can already state the characterisation of the MANS-semigroups.
	
	\begin{theorem}\label{thm-27}
		Let $S$ be a numerical semigroup with $\mathrm{msg}(S)=\{n_1<n_2<\cdots<n_e<n_{e+1}\}$ ($e\geq 2$) and let $S'=\langle n_1<n_2<\cdots<n_e\rangle$ with $\mathrm{Ap}(S',n_1)=\{w'(0), w'(1), \ldots, w'(n_1-1)\}$. Then $S$ is a MANS-semigroup if and only if
		\begin{enumerate}
			\item $S'$ is a MANS-semigroup,
			\item $n_e \bmod n_1 < n_{e+1} \bmod n_1$,
			\item and $w'(n_{e+1} \bmod n_1-1)<n_{e+1}<w'(n_{e+1}\bmod n_1)$.
		\end{enumerate}
	\end{theorem}
	
	\begin{proof}
		\textit{(Necessity.)} By Lemma~\ref{lem-25}, we know that $S'$ is a MANS-semigroup.
		
		Since $S$ is a MANS-semigroup and $n_e,n_{e+1}\in\mathrm{Ap}(S,n_1)$, we can state that $n_e \bmod n_1 < n_{e+1} \bmod n_1$.
		
		If $\mathrm{Ap}(S,n_1)=\{w(0), w(1), \ldots, w(n_1-1)\}$, since $S$ and $S'$ are MANS-semigroups, we have that $w'(i)=w(i)$ for all $i\in\{0,1,\ldots,n_{e+1} \bmod n_1-1\}$ and, consequently, $w'(n_{e+1} \bmod n_1-1)<n_{e+1}$.
		
		Finally, since $e_{n+1}\in \mathrm{msg}(S)$, we have that $e_{n+1}\not\in S'$, from which it follows that $n_{e+1}<w'(n_{e+1}\bmod n_1)$.
		
		\textit{(Sufficiency.)} That is Lemma~\ref{lem-24}.
	\end{proof}

	\subsection{Ap\'ery sets}\label{apery}
	
	Let $S$ and $S'$ be MANS-semigroups with $\mathrm{msg}(S)=\{n_1<n_2<\cdots<n_e\}$ and $\mathrm{msg}(S')=\{n_1<n_2<\cdots<n_e<n_{e+1}\}$. We now aim to construct $\mathrm{Ap}(S')$ from $\mathrm{Ap}(S)$.
	
	\begin{remark}\label{rem-24}
		Under the conditions of Lemma~\ref{lem-24}, let $t_{e+1}=n_{e+1}\bmod n_1$. Note that if we take $k=\left\lfloor \frac{n_1-1}{t_{e+1}} \right\rfloor+1$, then $k\geq 2$ and $kn_{e+1}\bmod n_1 \leq n_{e+1}\bmod n_1$. Since $S'=\langle n_1<n_2<\cdots<n_e<n_{e+1} \rangle$ is a MANS-semigroup, we deduce that $kn_{e+1}$ cannot appear as a summand in the elements of $\mathrm{Ap}(S',n_1)$.
	\end{remark}
	
	Let us take $K=\left\lfloor \frac{n_1-1}{t_{e+1}} \right\rfloor$, $A_h = \{ w(0),w(1),\ldots,w(n_1-1-ht_{e+1}) \}$, for $h\in\{0,1,\ldots,K\}$, and $\mathrm{Ap}(S,n_1)=\{w(0), w(1), \ldots, w(n_1-1)\}$. From the proof of Lemma~\ref{lem-24} and Remark~\ref{rem-24}, we have that $\mathrm{Ap}(S',n_1)$ is a subset of $B=A_0 \cup (n_{e+1}+A_1) \cup \cdots \cup (Kn_{e+1}+A_K)$ (where, as usual, if $x\in\mathbb{R}$ and $A \subseteq \mathbb{R}$, then $x+A = \{x+a \mid a\in A \} $). Thus if $\mathrm{Ap}(S',n_1)=\{w'(0), w'(1), \ldots, w'(n_1-1)\}$, then it is satisfied that
	\begin{itemize}
		\item $w'(i)=w(i)$ if $0\leq i\leq t_{e+1}-1$,
		\item $w'(i)=\min\{w(i),n_{e+1}+w(i-t_{e+1})\}$ if $t_{e+1}\leq i\leq 2t_{e+1}-1$,
		\item $w'(i)=\min\{w(i),n_{e+1}+w(i-t_{e+1}),2n_{e+1}+w(i-2t_{e+1})\}$ if $2t_{e+1}\leq i\leq 3t_{e+1}-1$,
		\item \ldots
		\item $w'(i)=\min\{w(i),n_{e+1}+w(i-t_{e+1}),2n_{e+1}+w(i-2t_{e+1}),\ldots, (K-1)n_{e+1}+w(i-(K-1)t_{e+1})\}$ if $(K-1)t_{e+1}\leq i\leq Kt_{e+1}-1$,
		\item $w'(i)=\min\{w(i),n_{e+1}+w(i-t_{e+1}),2n_{e+1}+w(i-2t_{e+1}),\ldots, (K-1)n_{e+1}+w(i-(K-1)t_{e+1}),Kn_{e+1}+w(i-Kt_{e+1})\}$ if $Kt_{e+1}\leq i\leq n_1-1$.
	\end{itemize}

	\begin{example}\label{exmp-28}
		Let $S=\langle n_1<n_2 \rangle$ and $S'=\langle n_1<n_2<n_3 \rangle$ be MANS-semigroups. Let us take $K=\left\lfloor \frac{n_1-1}{t_3} \right\rfloor$, with $t_3=n_3\bmod n_1$ and $0\leq t_3< n_1$. Then $\mathrm{Ap}(S,n_1) = \{0,n_2,2n_2,\ldots,(n_1-1)n_2\}$ and, by Proposition~\ref{prop-06}, we have that $\mathrm{Ap}(S',n_1)=\{w'(0), w'(1), \ldots, w'(n_1-1)\}$ with
		\begin{itemize}
			\item $w'(i)=in_2$ if $0\leq i\leq t_3-1$,
			\item $w'(i)=n_3+in_2$ if $t_3\leq i\leq 2t_3-1$,
			\item $w'(i)=2n_3+(i-2t_2)n_2$ if $2t_2\leq i\leq 3t_2-1$,
			\item \ldots
			\item $w'(i)=(K-1)n_3+(i-(K-1)t_2)n_2$ if $(K-1)t_2\leq i\leq Kt_2-1$,
			\item $w'(i)=Kn_3+(i-Kt_2)n_2$ if $Kt_2\leq i\leq n_1-1$.
		\end{itemize}
	\end{example}
	
	\begin{example}\label{exmp-29}
		Let $S_3=\langle 13, 27, 55 \rangle$. Then
		\[ \mathrm{Ap}(S_3,13) = \{0, 27, 54, 55, 82, 109, 110, 137, 164, 165, 192, 219, 220\}. \]
		Therefore, $S_3$ is a MANS-semigroup.
		
		If we take $n_4=96=7\times 13 +5$, then $82<96<109$ and, by Lemma~\ref{lem-24}, $S_4=\langle 13, 27, 55, 96 \rangle$ is a MANS-semigroup. Moreover, $K=\left\lfloor \frac{13-1}{5} \right\rfloor = 2$.
		
		To construct $\mathrm{Ap}(S_4,13)$, we consider the following table.
		\begin{center}
			\begin{tabular}{|c|c|c|}
				$w(i)$ & $96+w(i)$ & $2\cdot96+w(i)$ \\
				\hline \hline
				0 & - & - \\
				27 & - & - \\
				54 & - & - \\
				55 & - & - \\
				82 & - & - \\
				109 & 96+0=96 & - \\
				110 & 96+27=123 & - \\
				137 & 96+54=150 & - \\
				164 & 96+55=151 & - \\
				165 & 96+82=178 & - \\
				192 & 96+109=205 & 192+0=192 \\
				219 & 96+110=206 & 192+27=219 \\
				220 & 96+137=233 & 192+54=246 \\
				\hline
			\end{tabular}
		\end{center}
		Taking the minimum in each line, we conclude that
		\[ \mathrm{Ap}(S_4,13) = \{0, 27, 54, 55, 82, 96, 110, 137, 151, 165, 192, 206, 220 \}. \] 
	\end{example}

	\subsection{The tree of MANS-semigroups with multiplicity and ratio fixed}\label{tree}
	
	Note that if we fix the multiplicity value ($m>1$), then there are infinite MANS-semigrupos $S$ with $\mathrm{m}(S)=m$. In fact, by Proposition~\ref{prop-002}, we have that $S=\langle m, am+1 \rangle$ is a MANS-semigroup for any $a\in\mathbb{N}\setminus\{0\}$. Incidentally, $\mathbb{N}$ is the unique MANS-semigroup with multiplicity $m=1$. 
	
	However, if we fix the multiplicity ($m>1$) and the ratio ($r>2$), then the set $\mathcal{MA}(m,r)= \{ S \mid S \mbox{ is a MANS-semigroup, } \mathrm{m}(S)=m, \mbox{ and } \mathrm{r}(S)= r\}$ is finite. Indeed, by Corollary~\ref{cor-26}, it is clear that every element of $\mathcal{MA}(m,r)$ must contain the numerical semigroup $\langle m, r \rangle$. Now, since $\mathbb{N}\setminus \langle m, r \rangle$ is finite, we conclude that $\mathcal{MA}(m,r)$ has finitely many elements.
	
	Since we now want to find all the elements of $\mathcal{MA}(m,r)$, we will endow $\mathcal{MA}(m,r)$ with a tree structure.
	
	Recall that a \textit{directed graph} $G$ is a pair $(V,E)$ where $V$ is a non-empty set and $E$ is a subset of $\{ (u,v) \in V\times V \mid u\not= v \}$. The elements of $V$ and $E$ are called \textit{vertices} and \textit{edges}, respectively. A \textit{path}, of length $n$, connecting the vertices $u,v\in G$ is a
	sequence of distinct edges of the form $(v_0,v_1), (v_1,v_2), \ldots, (v_{n-1},v_n)$ such that $v_0=u$ and $v_n=v$.
	
	A directed graph $G$ is a \textit{tree} if there exists a vertex $r$ (known as the \textit{root} of $G$) such that for any other vertex $v\in G$, there exists a unique path connecting $v$ and $r$. Moreover, if $(u,v)$ is an edge of the tree, then $u$ is said to be a \textit{child} of $v$.
	
	To define the tree $G(\mathcal{MA}(m,r))$, we take $\mathcal{MA}(m,r)$ as the set of vertices and say that $(T,S)\in \mathcal{MA}(m,r)\times \mathcal{MA}(m,r)$ is an edge of $G(\mathcal{MA}(m,r))$ if and only if $\mathrm{msg}(S)=\mathrm{msg}(T) \setminus\{\mathrm{M}(T)\}$. 
	
	Given $S \in \mathcal{MA}(m,r)$, we define the following sequence: $S_0=S$ and
	\[ S_{n+1} = \left\{ \begin{array}{l}
		\langle \mathrm{msg}(S_n) \setminus \{\mathrm{M}(S_n)\} \rangle \mbox{ if } S_n\not=\langle m, r \rangle, \\[2pt]
		\langle m, r \rangle \mbox{ otherwise.}
	\end{array}\right. \] 
	
	From Lemma~\ref{lem-25}, we deduce the following result.
		
	\begin{proposition}\label{prop-30}
		If $S \in \mathcal{MA}(m,r)$ and $\{S_n\mid n\in \mathbb{N}\}$ is the sequence defined above, then $S_n\in \mathcal{MA}(m,r)$ for all $n\in\mathbb{N}$. Moreover, $S_k=\langle m, r \rangle$ for all $k\geq\mathrm{e}(S)-2$.
	\end{proposition}
	
	As a consequence of Proposition~\ref{prop-30}, we have the following result.
	
	\begin{proposition}\label{prop-31}
		$G(\mathcal{MA}(m,r))$ is a tree with root $\langle m, r \rangle$.
	\end{proposition}
	
	Note that, from its root, we can recurrently build a tree by connecting each vertex to its children through the corresponding edges. Thus, if we know the children of any vertex in $G(\mathcal{MA}(m,r))$, then we can build the tree and, moreover, find all the elements of $\mathcal{MA}(m,r)$.
	
	Let $S=\langle n_1<n_2<\cdots<n_e \rangle$ be a MANS-semigroup with $e\geq 2$. Moreover, let $\mathrm{Ap}(S,n_1) = \{w(0),w(1),\ldots,w(n_1-1)\}$. We will say that $n\in \mathbb{N}$ is \textit{suitably monotone for $S$} if it fulfils the next three conditions:
	\begin{enumerate}
		\item $n_e<n$.
		\item $n_e\bmod n_1 < n\bmod n_1$.
		\item $w(n \bmod n_1-1)<n<w(n\bmod n_1)$.
	\end{enumerate}
	
	From Lemma~\ref{lem-24} and Proposition~\ref{prop-31}, we deduce the result that allows us to recurrently build $G(\mathcal{MA}(m,r))$.
	
	\begin{theorem}\label{thm-32}
		 Let $S=\langle n_1<n_2<\cdots<n_e \rangle\in\mathcal{MA}(m,r)$ be any vertex of $G(\mathcal{MA}(m,r))$. Then the children of $S$ are the numerical semigroups $T_n=\langle n_1,n_2,\ldots,n_e,n\rangle$, where $n$ is suitably monotone for $S$. 
	\end{theorem}
	
	Let us illustrate the above theorem with two examples. In both of them, the number above the arrows corresponds to the modulo, with respect to the multiplicity, of the new minimal generator.
	
	\begin{example}\label{exmp-33a}
		Let $m=5$ and $r=6$. Then the tree $G(\mathcal{MA}(5,6))$ is given by Figure~\ref{fig-01}.
		\begin{figure}[ht]
			\centering
			\begin{tikzpicture}[scale=1]
				
				\node at (0.5,1) {$\langle 5,6 \rangle$};
				
				\draw [<-] (0.9,1.2) -- (2.375,1.7);
				\draw [<-] (0.9,1) -- (2.325,1);
				\draw [<-] (0.9,0.8) -- (2.325,0.3);
				
				\node at (1.6,1.6) {\scriptsize 2};
				\node at (1.6,1.125) {\scriptsize 3};
				\node at (1.6,0.725) {\scriptsize 4};
				
				\node at (3,1.75) {$\langle 5,6,7 \rangle$};
				\node at (3,1) {$\langle 5,6,13 \rangle$};
				\node at (3,0.25) {$\langle 5,6,19 \rangle$};
				
				\draw [<-] (3.625,1.75) -- (4.95,1.75);
				\draw [<-] (3.675,1) -- (4.8,1);
				
				\node at (4.25,1.875) {\scriptsize 3};
				\node at (4.25,1.125) {\scriptsize 4};
				
				\node at (5.75,1.75) {$\langle 5,6,7,8 \rangle$};
				\node at (5.75,1) {$\langle 5,6,13,14 \rangle$};
				
				\draw [<-] (6.525,1.75) -- (7.55,1.75);
				
				\node at (7.1,1.875) {\scriptsize 4};
				
				\node at (8.5,1.75) {$\langle 5,6,7,8,9 \rangle$};
			\end{tikzpicture}
			\caption{$G(\mathcal{MA}(5,6))$}
			\label{fig-01}
		\end{figure}
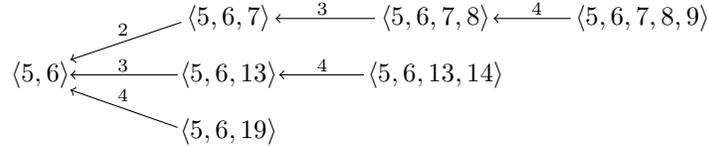
	\end{example}
	
	\begin{example}\label{exmp-33b}
		Let $m=5$ and $r=11$. Then the tree $G(\mathcal{MA}(5,6))$ is given by Figure~\ref{fig-02}.
		
		\begin{figure}[ht]
			\centering
			\begin{tikzpicture}[scale=1]
				
				\node at (0.5,4.5) {$\langle 5,11 \rangle$};
				
				\draw [<-] (0.9,4.75) -- (2.25,8.2);
				\draw [<-] (1,4.65) -- (2.25,5.6);
				\draw [<-] (1,4.45) -- (2.25,3.4);
				\draw [<-] (0.95,4.3) -- (2.25,2.15);
				\draw [<-] (0.9,4.15) -- (2.25,1.15);
				\draw [<-] (0.85,4) -- (2.25,0.4);
				
				\node at (1.55,6.8) {\scriptsize 2};
				\node at (1.55,5.25) {\scriptsize 2};
				\node at (1.55,4.175) {\scriptsize 3};
				\node at (1.55,3.5) {\scriptsize 3};
				\node at (1.55,2.95) {\scriptsize 4};
				\node at (1.55,2.45) {\scriptsize 4};

				\node at (3,8.325) {$\langle 5,11,12 \rangle$};
				\node at (3,5.75) {$\langle 5,11,17 \rangle$};
				\node at (3,3.25) {$\langle 5,11,23 \rangle$};
				\node at (3,2) {$\langle 5,11,28 \rangle$};
				\node at (3,1) {$\langle 5,11,34 \rangle$};
				\node at (3,0.25) {$\langle 5,11,39 \rangle$};
				
				\draw [<-] (3.75,8.45) -- (4.75,8.8);
				\draw [<-] (3.75,8.2) -- (4.75,7.85);
				\draw [<-] (3.75,5.875) -- (4.75,6.675);
				\draw [<-] (3.75,5.725) -- (4.75,5.4);
				\draw [<-] (3.75,5.575) -- (4.75,4.35);
				\draw [<-] (3.75,3.375) -- (4.75,3.5);
				\draw [<-] (3.75,3.125) -- (4.75,3.0);
				\draw [<-] (3.75,2.125) -- (4.75,2.25);
				\draw [<-] (3.75,1.875) -- (4.75,1.75);
				
				\node at (4.25,8.75) {\scriptsize 3};
				\node at (4.25,8.15) {\scriptsize 3};
				\node at (4.25,6.45) {\scriptsize 3};
				\node at (4.25,5.7) {\scriptsize 3};
				\node at (4.25,5.15) {\scriptsize 4};
				\node at (4.25,3.6) {\scriptsize 4};
				\node at (4.25,3.2) {\scriptsize 4};
				\node at (4.25,2.325) {\scriptsize 4};
				\node at (4.25,1.925) {\scriptsize 4};

				\node at (5.75,8.85) {$\langle 5,11,12,13 \rangle$};
				\node at (5.75,7.8) {$\langle 5,11,12,18 \rangle$};
				\node at (5.75,6.75) {$\langle 5,11,17,18 \rangle$};
				\node at (5.75,5.35) {$\langle 5,11,17,23 \rangle$};
				\node at (5.75,4.3) {$\langle 5,11,17,29 \rangle$};
				\node at (5.75,3.55) {$\langle 5,11,23,24 \rangle$};
				\node at (5.75,2.95) {$\langle 5,11,23,29 \rangle$};
				\node at (5.75,2.3) {$\langle 5,11,28,29 \rangle$};
				\node at (5.75,1.7) {$\langle 5,11,28,34 \rangle$};
				
				\draw [<-] (6.75,8.975) -- (7.5,9.15);
				\draw [<-] (6.75,8.725) -- (7.5,8.55);
				\draw [<-] (6.75,7.8) -- (7.5,7.8);
				\draw [<-] (6.75,6.875) -- (7.5,7.05);
				\draw [<-] (6.75,6.625) -- (7.5,6.45);
				\draw [<-] (6.75,5.475) -- (7.5,5.65);
				\draw [<-] (6.75,5.225) -- (7.5,5.05);
				
				\node at (7.15,9.2) {\scriptsize 4};
				\node at (7.15,8.75) {\scriptsize 4};
				\node at (7.15,7.925) {\scriptsize 4};
				\node at (7.15,7.1) {\scriptsize 4};
				\node at (7.15,6.675) {\scriptsize 4};
				\node at (7.15,5.7) {\scriptsize 4};
				\node at (7.15,5.25) {\scriptsize 4};
				
				\node at (8.75,9.2) {$\langle 5,11,12,13,14 \rangle$};
				\node at (8.75,8.5) {$\langle 5,11,12,13,19 \rangle$};
				\node at (8.75,7.8) {$\langle 5,11,12,18,19 \rangle$};
				\node at (8.75,7.1) {$\langle 5,11,17,18,19 \rangle$};
				\node at (8.75,6.4) {$\langle 5,11,17,18,24 \rangle$};
				\node at (8.75,5.7) {$\langle 5,11,17,23,24 \rangle$};
				\node at (8.75,5) {$\langle 5,11,17,23,29 \rangle$};
			\end{tikzpicture}
			\caption{$G(\mathcal{MA}(5,11))$}
			\label{fig-02}
		\end{figure}
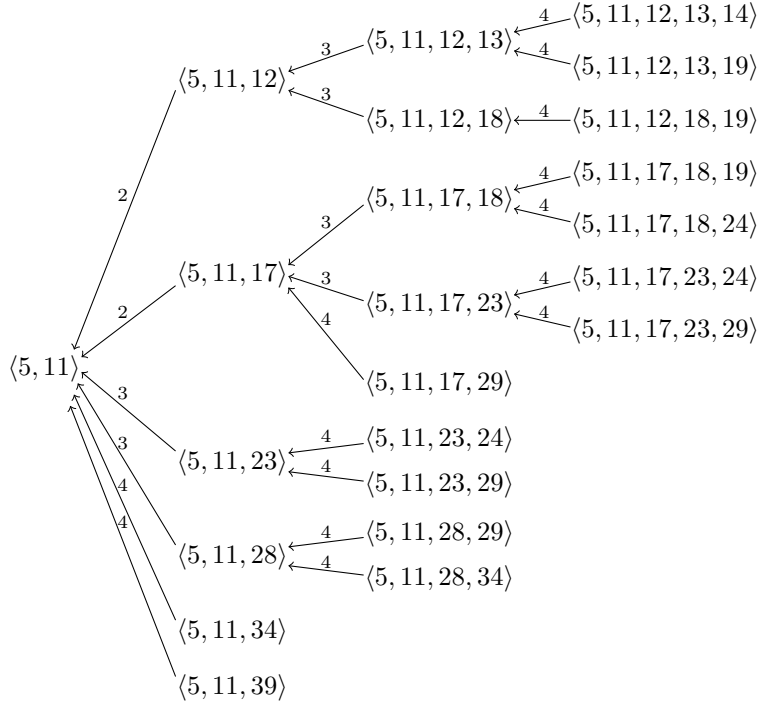
	\end{example}

	From two above examples, we observe that if $S$ is a MANS-semigroup such that $\mathrm{M}(S)\bmod{n_1}<n_1-1$, then $\mathrm{M}(S)+1$ is suitably monotone for $S$. As we show in the following result, this fact is not fortuitous.
	
	\begin{proposition}\label{prop-34}
		Let $S=\langle n_1<n_2<\cdots<n_e \rangle$ with $n_1\geq 3$ and $e\geq2$. If $S$ is a MANS-semigrupo such that $n_e\bmod{n_1}<n_1-1$, then $n_e+1$ is suitably monotone for $S$. 
	\end{proposition}
	
	\begin{proof}		
		Firstly, we observe that $n_e<n_e+1$ and, since $n_e\bmod{n_1}<n_1-1$, then $n_e\bmod{n_1} < n_e\bmod{n_1}+1 = (n_e+1)\bmod{n_1}$.
		
		Secondly, since $n_e\in\mathrm{msg}(S)$, if $\mathrm{Ap}(S,n_1)=\{w(0), w(1), \ldots, w(n_1-1)\}$, then $n_e=w(n_e\bmod{n_1})$ and, in consequence, $w((n_e+1)\bmod{n_1}-1)=w(n_e\bmod{n_1})=n_e<n_e+1$.
		
		At this moment, to prove that $n_e+1$ is suitably monotone for $S$, it remains to be seen that $n_e+1<w((n_e+1) \bmod{n_1})$. For this purpose, since $S$ is a MANS-semigroup, we note that $w(n_e\bmod{n_1}) < w((n_e+1)\bmod{n_1})$, that is, $n_e<w((n_e+1)\bmod{n_1})$. Therefore, $n_e+1 \leq w((n_e+1)\bmod{n_1})$.
		
		Suppose now that $n_e+1=w((n_e+1)\bmod{n_1})$. Then, we have that $n_e+1\in S$. Moreover, since  $n_e+1\not\in\mathrm{msg}(S)$, we can assert that there exist $i,j\in\{1,\ldots,e\}$ such that $i<j$ and $n_e+1=w(i)+w(j)$. Thus $n_e=w(i)+w(j-1)$ or $n_e=w(i-1)+w(j)$ (recall again that $n_e=w(n_e\bmod{n_1})$). However, since $n_e\in\mathrm{msg}(S)$, it is not possible that $n_e=w(i)+w(j-1)$. Furthermore, $n_e=w(i-1)+w(j)$ only if $i-1=0$ and $j=e$. In such a case, $n_e+1=w(i)+w(j)=w(1)+w(e)=w(1)+n_e$, which is a contradiction because $w(1)=n_2>n_1\geq3$. Thus, we conclude that $n_e+1<w((n_e+1)\bmod{n_1})$.
	\end{proof}
	
	An immediate consequence of Proposition~\ref{prop-34} is the following result.
	
	\begin{corollary}\label{cor-35}
		If $m\in\mathbb{N}\setminus\{0,1,2\}$ and $a\in\mathbb{N}\setminus\{0\}$, then the numerical semigroups $S_1=\langle m, am+1 \rangle$, $S_2=\langle m, am+1, am+2 \rangle$, \ldots, $S_{m-1}=\langle m, am+1, \ldots, am+(m-1) \rangle$ belong to the tree $G(\mathcal{MA}(m,am+1))$. Moreover, $S_{i+1}$ is a child of $S_i$ for $i\in\{1,2,\ldots,m-2\}$.
	\end{corollary}
	
	To finish the subsection, we will see that it is possible to compute the number of children of a MANS-semigroup $S\in\mathcal{MA}(m,r)$ if we know its Frobenius number $\mathrm{F}(S)$. Indeed, if $\mathrm{Ap}(S,m) = \{w(0),w(1),\ldots,w(m-1)$, then $n\in\mathbb{N}$ could be suitably monotone for $S$ whenever $n\bmod{m}>\mathrm{M}(S)\bmod{m}$. Therefore, we will only find suitably monotone elements in intervals $(w(i),w(i+1))$ such that $\mathrm{M}(S)\leq i \leq m-2$ and $w(i+1)-w(i)>m$. Moreover, there will be $\frac{w(i+1)-w(i)-1}{m}$ suitably monotone elements in the interval $(w(i),w(i+1))$ (precisely, the numbers congruent to $w(i+1)$ modulo $m$). From here, the number of children will be given by the expression $\sum_{i=\mathrm{M}(S)}^{m-2}\frac{w(i+1)-w(i)-1}{m}$. From a simple computation, it follows the following result.
	
	\begin{proposition}\label{prop-36}
		A numerical semigroup $S\in\mathcal{MA}(m,r)$ has $\left\lfloor \frac{\mathrm{F}(S)-\mathrm{M}(S)}{m} \right\rfloor +1$ children in the tree $G(\mathcal{MA}(m,r))$.
	\end{proposition}
	
	Let us see two illustrative examples of the above proposition.
	
	\begin{example}\label{exmpl-36a}
		It is well-known that $\mathrm{F}(\langle 5, 11 \rangle)=39$. Therefore, the number of children of $S=\langle 5, 11 \rangle$ is $\left\lfloor \frac{39-11}{5} \right\rfloor +1=6$ (in $G(\mathcal{MA}(5,11))$). Specifically, as Figura~\ref{fig-02} shows, $S$ has
		\begin{itemize}
			\item two children in $(w(1),w(2))=(11,22)$ (for $n=12$ and $n=17$),
			\item two children in $(w(2),w(3))=(22,33)$ (for $n=23$ and $n=28$),
			\item and two children in $(w(3),w(4))=(33,44)$ (for $n=34$ and $n=39$).
		\end{itemize}
		For those intervals, it is clear that $\frac{22-11-1}{5}=\frac{33-22-1}{5}=\frac{44-33-1}{5}=2$.
	\end{example}
	
	\begin{example}\label{exmpl-36b}
		Let $S$ be the numerical semigroup given by $\langle 7,15,16 \rangle$. Then $\mathrm{Ap}(S,7)=\{w(0)=0,w(1)=15,w(2)=16,w(3)=31,w(4)=32,w(5)=47,w(6)=48\}$ (that is, $S$ is a MANS-semigroup) and $\mathrm{F}(S)=41$. Thus, $S$ has $\left\lfloor \frac{41-16}{7} \right\rfloor +1=4$ children in $\mathcal{MA}(7,15)$. Indeed, $S$ has
		\begin{itemize}
			\item two children in $(w(2),w(3))=(16,31)$ (for $n=17$ and $n=24$),
			\item and two children in $(w(4),w(5))=(32,47)$ (for $n=33$ and $n=40$).
		\end{itemize}
	\end{example}

	\section*{Acknowledgements}
	
	Both authors are supported by Proyecto de Excelencia de la Junta de Andaluc\'{\i}a Grant Number ProyExcel\_00868 and by the Junta de Andaluc\'{\i}a Grant Number FQM-343.

\end{document}